\documentclass[12pt]{amsart}
\usepackage{fullpage}
\usepackage{color}
\usepackage{url}
\usepackage{graphicx} %to insert graphics
\usepackage{tikz}  % to build pictures in LaTeX
\usepackage[all]{xy} %more for building pictures
\usepackage[affil-it]{authblk} 
\usepackage{subfigure} %Self explanatory
\usepackage{placeins} %So I can use the Floatbarrier
\usepackage{amssymb}
\usepackage{amsmath}
\usepackage{enumerate}
\usepackage{enumitem}
\usepackage{xcolor}
\usepackage{mathtools}
\usepackage{bbding}
\usepackage{hyperref}
\usepackage{titling}

\newtheorem{prop}{Proposition}[section]
\newtheorem*{rmk}{Remark}

\newtheorem{ex}[prop]{Example}
\newtheorem{thm}[prop]{Theorem}
\newtheorem{cor}[prop]{Corollary}
\newtheorem{defn}[prop]{Definition}%[section]
\newcommand{\rarrow}{\rightarrow}
\newcommand{\ep}{\epsilon}
\newcommand{\om}{\omega}
\newcommand{\ZZ}{\mathbb{Z}}

\newcommand{\RR}{\mathbb{R}}

\newcommand{\id}{\text{id }}

\newcommand{\ol}{\overline}

\newcommand{\cl}{\overset{cl}}
\newcommand{\es}{\emptyset}

\newcommand{\konf}{\mathfrak{K}}

\newcommand{\Int}{\text{int}}

\hypersetup{
    colorlinks=true,
    linkcolor=blue,
    filecolor=magenta,      
    urlcolor=cyan,
    pdftitle={Sharelatex Example},
    bookmarks=true,
    pdfpagemode=FullScreen,
    hyperindex=true,
}

\title{The Semicontinuity of Attractors for Closed Relations on Compact Hausdorff Spaces}
\author{Shannon Negaard-Paper, PhD}
\date{October 2019}

\begin{document}

\maketitle

%%%%%%%%%%%%%%%%%%%%%%%%%%%%%%%%%%%%%%%%%%%%%%%%%%%%%%%%%%
%%%%%%%%%%%%       CLOSENESS OF RELATIONS     %%%%%%%%%%%%%%%%%%%%%%%%%%%
%%%%%%%%%%%%%%%%%%%%%%%%%%%%%%%%%%%%%%%%%%%%%%%%%%%%%%%%%%
%\chapter{Semicontinuity of attractors}\label{SemiContAtt}
\section*{Abstract}
We show that attractors are semicontinuous for closed relations on compact Hausdorff spaces. Semicontinuity is what guarantees that small changes to a system do not result in massive growth of certain features, notably attractors. That is, there is a certain preservation of structure. When it comes to flows, semiflows, and maps, it is well established that attractors are semicontinuous. In \ref{bk:McGeheeAttractors}, relations were established as a way to generalize maps, and a formal definition of attractors was established. Relations (in the dynamical systems sense) represent discrete time systems, which may lack uniqueness (or existence) in forward time. %This paper pro that those attractors are semicontinuous.% represent discrete time systems which may lack uniqueness in forward time. McGehee

\section{Introduction}

We start in a well-known setting - maps, and save discussion of relations for later. Let $f$ be a map over a topological space $X$. Then, an attractor $A$ is a compact invariant set, which has some neighborhood $U$ where 
	\[ A = \om(U;f) = \bigcap\limits_{n\ge 0}\ol{ \bigcup_{k\ge n}   f^n(U)   } .\]
Attractors are a fundamental type of invariant set, and they play a large role in understanding the structure of any dynamical system. When analyzing the structure of a system, we usually look first at ultimite behavior, and therefore find attractors. These help us define repeller duals, connecting orbits, etc. \ref{bk:Conley},\ref{bk:Mischaikow}.

We therefore wish to know when they persist; that is, say we have a dynamical system with a an attractor of interest, then how much can we change a dynamical system and still have an attractor in the same region of our space? Even more precisely, when do we have semicontinuity of attractors? Semicontinuity of attractors guarantees the preservation of a fundamental structural element: if a system has an attractor, then there are nearby systems with their own nearby attractors.
Say a system $f$ over $X$ has an attractor $A$. Semicontinuity means that given some goal neighborhood $U$ of $A$, we can find a bounding neighborhood of $f$, such that any system $g$ in that neighborhood (a system that is ``close to" $f$) also possesses an attractor $A'$ in $U$ (so, close to $A$). 
This puts a limit on the {\it growth} of attractors, as we slowly change a system. What we are not guaranteed is a limit on the shrinkage of attractors, as Example \ref{ex:ShrinkingAttMap} demonstrates.
%\shannon{quick example with a map?}

\begin{ex}\label{ex:ShrinkingAttMap}
Consider the family of (continuous) maps $\frak{F}_\alpha=\{f(x)=-x^3 + 3x^2 + x - \alpha\}$. For the system where $\alpha = 4$, there are two fixed points: $x=-1,2$. In fact, $A=[-1,2]$ is an attractor. See the phase space diagram in Figure \ref{fig:ShrinkingAttMap}, in which the attractors are orange. If we decrease $\alpha$ by a little, the attractor shifts slightly (until we hit $\alpha=0$). We can put a neighborhood $U$ (the green box in the figure) around $A$, and this will dictate how much we can move $\alpha$ in either direction. To the left, the attractor inside $U$ will still be a closed interval, with the bounds shifting slowly. If we let $\alpha>4$, however, the largest invariant set inside $U$ is a single point, close to $x=-1$.

\FloatBarrier
\begin{figure}
\centering
\includegraphics[width=.3\textwidth]{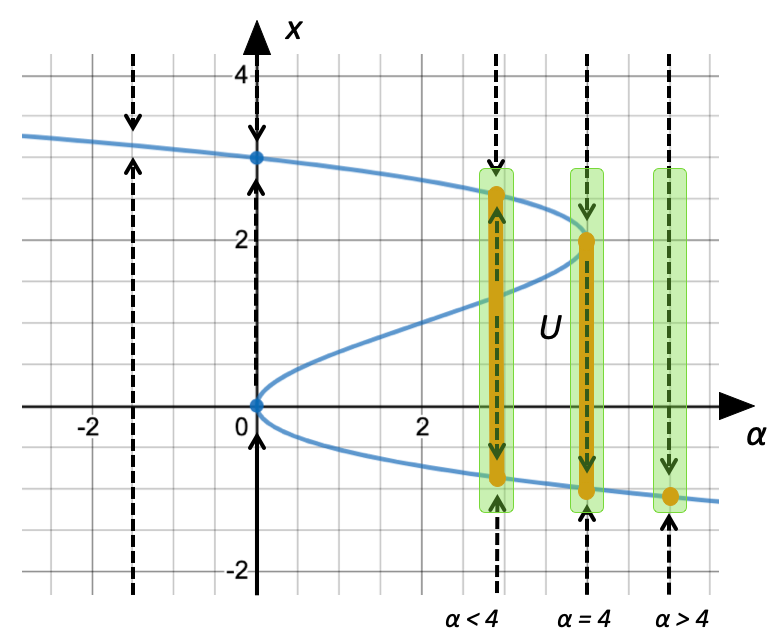}\label{fig:ShrinkingAttMap}
\caption{Phase space diagram for $f(x)=-x^3 + 3x^2 + x - \alpha$ in Example \ref{ex:ShrinkingAttMap}.}
\end{figure}
\end{ex}

An important breakthrough, in answering the question of semicontinuity of attractors for flows and maps, arises from an idea of C. Conley \ref{bk:Conley},\ref{bk:Mischaikow}. He shifted the focus onto {\it attractor blocks} and established associations between attractor blocks and the attractors inside them. We will use the same tool: attractor blocks, but in a setting where systems might lack uniqueness in forward time.

%\begin{rmk}\label{rmk:Empty}

%\end{rmk}

%$\blacklozenge$ 
In \ref{bk:McGeheeAttractors}, R. McGehee established the use of relations for generalizing discrete dynamical systems, lacking uniqueness in either forward or backward time. %\shannon{In Section \ref, we establish the semicontinuity of attractors for closed relations over compact Hausdorff spaces. Then, for compact metric spaces, we give a size of neighborhood of a closed relation, in which all closed neighborhoods have attractor... gah. Fix wording later.}
 %$\blacklozenge$ 
 Relations represent part of a natural progression - from invertible maps, to all maps (unique images in forward, but not necessarily backward, time), to relations.
In \ref{bk:McGeheeAttractors}, a great number of foundational ideas and terminology were established. We review the ones we need in Sections \ref{sec:Defns} and \ref{sec:Attractors}, with the latter focusing on definitions and theorems related to attractors. %In particular, we define attractors for relations in Section \ref{sec:Attractors}.%e pay special attention to the notion of an attractor, defined in Section \ref{sec:Attractors}.

Relations have proven a fruitful tool (\ref{bk:McGeheeSander},\ref{bk:McGeheeWiandt},\ref{bk:Me}), but one thing that had not previously been addressed is the semicontinuity of attractors, occurring in systems defined by closed relations. This brings us to the main result, which will be proven in Section \ref{sec:Semicontinuity}.%, we prove the main result:
\begin{quotation}
%\begin{thm}[Attractors over compact Hausdorff spaces are semicontinuous]\label{thm:AttSemiCont}
{\bf Theorem \ref{thm:AttSemiCont}}
{\it  Let $X$ be a compact Hausdorff space, and let $f\cl{\subset} X\times X$ be a closed relation over $X$. Suppose also that $A\subset X$ is a nonempty attractor for $f$. Given any open neighborhood $U$ of $A$, there is an open neighborhood $V$ of $f$ such that any other closed relation $g\cl{\subset} V$ has an attractor $A'\subset U$. }%\end{thm}
 \end{quotation}
 That is, attractors for closed relations on compact Hausdorff spaces are semicontinuous.

%$\circ$ This is the notion of semicontinuity. What we mean is that if a system has an attractor, then there are nearby systems with their own nearby attractors (not necessarily non-empty). 

%This   
%$\lozenge$ Attractors are a fundamental type of invariant set, and they play a large role in understanding the structure of any dynamical system. We also wish to know when they persist; that is, say we have a dynamical system with a particular attractor of interest, then how much can we change a dynamical system and still have an attractor in the same region of our space? An important breakthrough in answering this question for flows and maps arises from an idea of C. Conley. We shift the focus onto attractor blocks and establish associations between attractor blocks and the attractors inside of them \ref{bk:Conley},\ref{bk:Mischaikow}. 

%In \ref{bk:McGeheeAttractors}, R. McGehee rigorously defined attractors and attractor blocks for relations. Finally, in Section \ref{sec:Semicontinuity}, we will prove the semicontinuity of attractors for closed relations over compact Hausdorff spaces. We will also give a more explicit proof of the semicontinuity of attractors for closed relations over compact metric spaces - a special case, which allows us do define distances.

%metric spaces and briefly discuss one reason why this semicontinuity is so useful.

%%%%%%%%%%%%%%%%%%%%%%%%%%%%%%%%%%%%%%%%%%%%%%%%%%
%%%%%%%%%%%%%%%%%%%%%%%%%%%%%%%%%%%%%%%%%%%%%%%%%%
%%%%%%%%%%   DEFINITIONS
\section{Review definitions \& theorems}\label{sec:Defns}

%Condense the following, borrowing heavily from \ref{bk:McGeheeAttractors} to establish its importance:
%\begin{itemize}
%\item Definition of Relation  \textcolor{purple}{\checkmark}
%\item Explanation of why we care about attractors
%\item Definition of Attractor \textcolor{purple}{\checkmark}
%\item Explanation of why attractor blocks are so useful.
%\item Definition of attractor block
%\end{itemize}

We begin with a review of the definition of a relation. We'll expand on the motivation shortly.

\begin{defn}
A {\it relation} $f$ on a space $X$ is a subset $f\subset X\times X$.
\end{defn}

The graph of a map $F:X\rarrow X$ is a relation $f\subset X\times X$, and relations allow for us to include situations, which lack uniqueness in ``forward time." Simple examples lack uniqueness. For instance, let $X=\RR$ be our space and let $y^2=x$. Then $y=\pm\sqrt{x}$, which is not a function of $x$, but its graph $f=\{(x,y): y^2=x\}$ in $\RR\times\RR$ is a relation. Thus, relations are the natural objects to consider for the purpose of generalizing maps. More arguments about their usefulness can be found in \ref{bk:McGeheeAttractors}. We focus on closed relations because the graph of any continuous map $F:X\overset{cont}{\rarrow} X$ is a closed relation $f\overset{cl}{\subset} X\times X$. Thus, closed relations serve as a generalization of continuous maps.

We know how to find the image under a map, as well as how to iterate a map, so as to move forward in discrete time. We review how these concepts transfer to relations. %For more discussion and examples, see \ref{bk:McGeheeAttractors}

%{\bf Image}
\begin{defn}
Let $f$ be a relation%\footnote{We use $f$, despite the fact that these rarely correspond to functions. This is for historic purposes and to encourage a perspective, which treats relations as set maps.} 
on $X$, and let $S \subset X$. The {\it image} of S under $f$ is the set
\[ f(S)\equiv \{y\in X: \hbox{ there is some }x\in S\hbox{ satisfying }(x,y)\in f\}.\]
\end{defn}

Specifically, we will frequently care about finding the image of a single point. We use the abuse of notation
		\[ f(x) = f(\{x\}) = \{ y\in X : (x,y)\in f\}.\]

\begin{defn}If $f$ is a relation over $X$, then for $n\in\ZZ_{\ge0}$ $f^n$ is also a relation over $X$ defined by
	\[f^0=\id_X=\left\{(x,x):x\in X\right\},\]
	\[f^n=\left\{(x,z): z=f(f^{n-1}(x))\right\} \hbox{ if }n>0.\]
%Therefore, $f^n(x)=\{z: (x,z)\in f^n\}$ and if $S\subset X$, \[f^n(S)=\{z: \hbox{ there is some }x\in S\hbox{ such that }(x,z)\in f^n\}.\]
\end{defn}

Because $f^n$ is also a relation, we already know how to define the images of forward time iterations. Also, if $f$ is closed, then so is $f^n$ (for $n\in\ZZ_{\ge0}$). This is a quick result: the identity is closed, and the rest is a result of Theorem 2.2 from \ref{bk:McGeheeAttractors}, which states that if $f$ and $g$ are closed relations over a space $X$, then so is $f\circ g$.

For reasons explored in further depth in \ref{bk:McGeheeAttractors}, we only consider composition for $n\ge0$, and in order to move in backward time we consider the transpose of relations.

\begin{defn}
Let $f\subset X\times X$ be a relation on $X$. Then, {\it f transpose} is defined as
		\[f^* = \left\{(x,y):(y,x)\in f\right\}.\]
\end{defn}

\begin{ex}
Consider again the example where $y^2 = x$ ($x,y\in X=\RR$). We'll build it from the transpose of another relation. Let $y=G(x)=x^2$, a map whose graph is \[g=\{(x,y):y=x^2\}\subset \RR\times\RR.\] Notice that $g$ is already a relation (and the graph of a function). We simply take its transpose $f=g^*=\left\{(x,y):x=y^2\right\}.$ Moving ``backward" in time would involve iterating $f$. We look at the images under $f$:
		\begin{align*}
				f(x)=&\emptyset \hbox{ if } x<0, \\
				f(0)=&\{0\}, \hbox{ and }\\
				f(x)=&\{\pm\sqrt{x}\}\hbox{ if }x>0.
		\end{align*} 
Relations allow for non-unique images and images that are empty.
\end{ex}

The transpose is not a true inverse. In general, one cannot take a relation $f$ (over $X$) and its transpose $f^*$ and combine them to get the identity relation. However, if $f$ is the graph of an invertible function $F:X\rarrow X$, then $f^*$ is the graph of $F^{-1}$. This is not our current area of exploration. For further details, once again see $\ref{bk:McGeheeAttractors}$.

%%%%%%%%%%%%%%%%%%%%%%%%%%%%%%%%%%%%%%%%%%%%%%%%%%
%%%%%%%%%%%%%%%%%%%%%%%%%%%%%%%%%%%%%%%%%%%%%%%%%%
%%%%%%%%%%   ATTRACTORS
\section{Attractors}\label{sec:Attractors}

One of the fundamental formations we look for in dynamical systems are attractors. These have long held definitions in systems with uniqueness in forward time: those defined by flows, semi-flows, and maps (both invertible and not). In \ref{bk:McGeheeAttractors}, the remaining definitions were established.

\begin{defn}
A set $S\subset X$ is called {\it invariant} under the relation $f$ if $f(S)=S$.
\end{defn}

\begin{defn}%[from \ref{bk:McGeheeAttractors}]
A set $A\subset X$ is called an {\it attractor} for relation $f\subset X\times X$ if 
\begin{enumerate}
	\item $f(A)=A$ (so $A$ is invariant), and
	\item there exists a neighborhood $U$ of $A$ such that \[\om(U)=A.\]
\end{enumerate}
\end{defn}

The omega limit set is defined as below.

\begin{defn}
%Omega limit set
If $f$ is a relation over $X$ and $S\subset X$, then the {\it omega limit set} of $S$ under $f$ is
	\[ \om(S;f)\equiv \bigcap \frak{K}(S;f),\]
where 
	\begin{align*} \frak{K}(S;f)\equiv& \{ K: K \hbox{ is a closed confining set satisfying}\\
							& f^n(S)\subset K \hbox{ for some }n\ge0\}.\end{align*}
We may abbreviate $\konf(S)$, $\konf_f$ or even $\konf$ when the $S$ or $f$ is understood from context. Likewise, $\om(S;f)$ may be abbreviated to $\om(S)$ if $f$ is clear from context.
\end{defn}

For an explanation of why this differs from the usual omega limit set for maps, as well as when these definitions agree, see \ref{bk:McGeheeAttractors}. %or \ref{bk:MyDiss}.

Attractors are useful objects to find, but what happens when the underlying system changes, even slightly?

\begin{ex}\label{ex:Attractor}
Let $X=[-5,5]$ and consider the system defined by $f(x)=\frac{x-1}{2}$. That is, the relation is $f\subset X\times X = \{(x,y):y=(x-1)/2\}$. This system has one attractor (the only equilibrium) at $x=-1$. Let's change this only slightly: $g(x)=\frac{x-1.1}{2}$, with an attractor at $x=-1.1$. These attractors are clearly linked, but they don't share a location. If we were looking for attractors in the family of systems \[\mathfrak{F} = \left\{ f_\alpha \subset X\times X: (x,y)\in f_\alpha \iff y=\frac{x+\alpha}{2} \hbox{ for all }\alpha\in\RR\right\},\] we would need to look at new locations.%, and if $|\alpha|>5$, we lose the attractor entirely (it exists outside $X$).
They are not robust to parameter changes. 
\end{ex}

So, we find a set that is linked to attractors, but which is robust to parameter changes (or better yet, which continues).
For maps and flows, C. Conley proposed the use of {\it attractor blocks} \ref{bk:Conley},\ref{bk:Mischaikow}, and R. McGehee extended this notion to relations \ref{bk:McGeheeAttractors}. Attractor blocks are often robust to parameter changes.

\begin{defn}
Given topological space $X$ and relation $f\subset X\times X$, $B\subset X$ is an attractor block if
			\[ f(\ol{B})\subset \Int (B).\]
\end{defn}

Another definition will be useful in Section \ref{sec:Semicontinuity}, so we include it here.

\begin{thm}[Lemma 7.7 from \ref{bk:McGeheeAttractors}]\label{thm:ABAltDefn}
A set $B\subset X$ is an attractor block for relation $f$ if and only if $f\cap\left(\ol{B}\times\ol{B^C}\right) = \emptyset.$
\end{thm}

\begin{proof}
Let $f \subset X\times X$ be a relation. Then
	\begin{align*}
		f \hbox{ is an attractor block} \iff& f(\ol{B})\subset \Int(B) \\
							\iff& f(\ol{B})\cap \ol{B^C} = \es \\
							\iff& f\cap (\ol{B}\times\ol{B^C}) = \es.
	\end{align*}
\end{proof}

%This equivalence works because we really just need $\ol{f(B)}\cap\ol{B^C}=\es$.

Attractor blocks are useful but require translation. As was done for flows (see \ref{bk:Conley}, \ref{bk:Mischaikow}), one needs to connect attractors to attractor blocks, and vice versa. Some assumptions are necessary, as you'll see in Theorems \ref{thm:BImpA} and \ref{thm:AandNImpB}. Given an attractor block $B$, we can guarantee an attractor inside $A=\om(B)\subset B$ (see Figure \ref{fig:BImpA}). Such an attractor is said to be {\it the attractor associated to $B$}. For the other direction, we require an attractor $A$ and a bounding neighborhood; given those, we can guarantee the existence of an attractor block inside the bounding neighborhood, which contains said attractor in its interior. Then $B$ is {\it an attractor block associated to $A$}. %These are not necessarily unique. %\shannon{Rough paragraph - refine and add a reference to the other picture. After the picture is fixed, use $N$ instead of $U$, if referenced.}%An attractor

%\begin{figure}\centering%[h]\label{Fig:CE1}
%  \includegraphics[width=3.5in]{AttNbhdImpAttr.png}\label{fig:MapEx1f}
%  \caption{}
%\end{figure}

\begin{figure}
\centering
\subfigure[An attractor block $B$ is guaranteed an attractor $A$ associated to $B$.]
{
	\includegraphics[width=.4\textwidth]{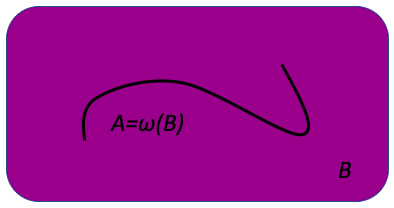}\label{fig:BImpA}
}
\subfigure[Given any neighborhood $U$ of an attractor $A$, we are guaranteed an attractor block $B\subset U$ associated to $A$.]
{
	\includegraphics[width=.4\textwidth]{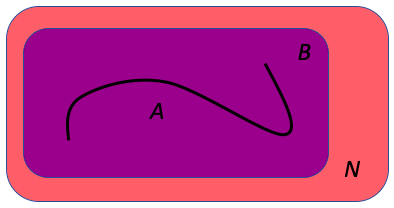}\label{fig:AandNImpB}
}
%\caption{\shannon{Fix the picture, so it's $N$ instead of $U$ - in powerpoint.}}
\end{figure}

%Theorem that attractor block has unique attractor
\begin{thm}[Theorem 7.2 from \ref{bk:McGeheeAttractors}]\label{thm:BImpA}
If $f$ is a closed relation on a compact Hausdorff space and if $B$ is an attractor block for $f$, then $B$ is a neighborhood of $\om(B)$ and, hence, $\om(B)$ is an attractor for $f$.
\end{thm}

\begin{thm}[Theorem 7.3 from \ref{bk:McGeheeAttractors}]\label{thm:AandNImpB}
If $f$ is a closed relation on a compact Hausdorff space, if $A$ is an attractor for $f$, and if $N$ is a neighborhood of $A$, then there exists a closed attractor block $B$ for $f$ such that $B\subset N$ and $\om(B)=A$.
\end{thm}

%Theorem that attractor + neighborhood has an attractor block.

%Rmk: attractor blocks not necessarily unique.
\begin{rmk}
Given an attractor $A$ and a bounding neighborhood $N$, we are able to acquire an attractor block associated to $A$ and contained in $N$, but there is no guarantee of uniqueness. At the end of Example \ref{ex:AttractorBlock}, we'll see this. %\shannon{Is there an example of this we can reference / quickly describe?}
\end{rmk}

With Theorems \ref{thm:BImpA} and \ref{thm:AandNImpB}, we know that we can translate in both directions, making attractor blocks useful tools for understanding attractors. 

\begin{ex}\label{ex:AttractorBlock}
We revisit the relation family from Example \ref{ex:Attractor}:
	\[\mathfrak{F} = \left\{ f_\alpha \subset X\times X: (x,y)\in f_\alpha \iff y=\frac{x+\alpha}{2} \hbox{ for all }\alpha\in\RR\right\},\]
where $X=[-5,5]$. We made $X$ compact ($\RR$ is already Hausdorff) so that Theorems \ref{thm:BImpA} and \ref{thm:AandNImpB} apply. Recall that we started with $f_{-1}=\{(x,y): y=\frac{x-1}{2}\}$, which has an attractor at $A_{-1} = \{-1\}\subset X$. Let $N$ be the $\ep$-ball around $A_{-1}$ where $\ep=0.3$, and choose $B = \ol{N}_{0.2}(A_{-1}) = [-1.2, -0.8]\subset \Int(N)$. This is an attractor block because 
	\[f(\ol{B})=f(B) = [-1.1,-0.9]\subset \Int(B).\]
Then $$\mathfrak{F}_B = \left\{ f_\alpha \subset X\times X: (x,y)\in f_\alpha \iff y=\frac{x+\alpha}{2}\hbox{ for } \alpha\in (-1.2,-0.2)\right\}=\mathfrak{F}\mid_{\alpha\in(-1.2,-0.2)}$$
represents a sub-family of relations $f_\alpha$ with attractors $A_\alpha=\{\alpha\}=\om(B)\subset B$. Therefore, these attractors are also associated (when paired with the correct relation) to $B$.

There are more relations in, which share $B$ as an attractor block. These relations would have an attractor ``close" to $A_{-1}$ (because they're within $B$, which is in turn in $N$). In Section \ref{sec:Semicontinuity} we will define some criteria for finding more such relations.

Furthermore, $A_{-1}$ was the attractor, and $N$ the given neighborhood, but we had many choices for $B$. In this example it was easy to find an attractor block. Any subset $V\subset X$, which satisfied $V\subset\Int(N)$ and $A_{-1}\subset\Int(V)$ would be an attractor block for $f_{-1}$ associated to $A_{-1}$. The attractor blocks are not always so simple to find (especially in higher dimensions), and they are not necessarily unique.
\end{ex}

%%%%%%%%%%%%%%%%%%%%%%%%%%%%%%%%%%%%%%%%%%%%%%%%%%
%%%%%%%%%%%%%%%%%%%%%%%%%%%%%%%%%%%%%%%%%%%%%%%%%%
%%%%%%%%%%   SEMICONTINUITY 
\section{Semicontinuity of attractors for relations}\label{sec:Semicontinuity}

Because we rely on Theorems \ref{thm:BImpA} and \ref{thm:AandNImpB}, we need $X$ to be compact and Hausdorff. In this setting, however, attractors for closed relations are semicontinuous. This gives us a kind of breathing room. If a relation $f$ is used as a model, and we know it has a non-empty attractor in a given region, then even if we need to adjust our relation (within reason) to a nearby relation $g$, then $g$ also has an attractor near where we expect one.

\begin{thm}[Attractors over compact Hausdorff spaces are semicontinuous]\label{thm:AttSemiCont}
Let $X$ be a compact Hausdorff space, and let $f\cl{\subset} X\times X$ be a closed relation over $X$. Suppose also that $A\subset X$ is a nonempty attractor for $f$. Given any open neighborhood $U$ of $A$, there is an open neighborhood $V$ of $f$ such that any other closed relation $g\cl{\subset} V$ has an attractor $A'\subset U$. \end{thm}

\begin{proof}
Assume $X$, $f$, and $A$ are as described. Let $U$ be any open neighborhood of $A$. By Theorem \ref{thm:AandNImpB}, there is a closed attractor block $B\subset U$, associated to $A$ (meaning $\om(B;f)=A$). Let $V=(\ol{B}\times\ol{B^C})^C$. Then $V$ is open, and for any $g\cl{\subset} V$,
	\[       g\cap (\ol{B}\times\ol{B^C}) \subset V \cap (\ol{B}\times\ol{B^C}) = \es.\]
By Theorem \ref{thm:ABAltDefn}, this means $B$ is also an attractor block for the relation $g$. We are thus guaranteed that $A'=\om(B;g)\subset B\subset U$ is an attractor for $g$ (Theorem \ref{thm:BImpA}).
\end{proof}

\begin{rmk}
Due to the choice of $B$, $V$ is not necessarily unique (see Example \ref{ex:AttractorBlock}). Also, there is no guarantee that $A'=\om(B;g)$ is non-empty (see Example \ref{ex:EmptyAttractor}).
\end{rmk}

\begin{ex}\label{ex:EmptyAttractor}
Let $X$ be any compact Hausdorff space, let $f\cl{\subset}X\times X$ be a closed relation with attractor $A$, and let $B$ a choice of attractor block associated to $A$. Then, $V=(\ol{B}\times\ol{B^C})^C$. We choose the simplest closed relation $g=\es\cl{\subset}X\times X$. Then, $g\subset V$, so $B$ is also an attractor block for $g$. In this case, $\om(B;g) = \es\subset B$, which is an attractor associated to $B$.
\end{ex}

This result was originally formulated in compact metric spaces. It is worth considering some implications in this more concrete setting.

\begin{cor}\label{cor:RelnBubbleAttBlkMet}
Let $f\overset{cl}\subset X\times X$ be a closed relation with $B$ an attractor block for $f$. Let $\delta = d\left( f, \ol{B}\times\ol{B^c} \right)>0$. Then any closed relation $g\cl{\subset}\Int(N_\delta(f)) $%\ol{N_\ep}(f)\cap (\ol{B}\times\ol{B^c})=\es$ for any $\ep<\delta$.
has an attractor $A'=\om(B;g)\subset B$.
\end{cor}

\begin{proof}
Let everything be as in the hypothesis. Then $g\subset (\ol{B}\times\ol{B^C})^C = V$. By Theorem \ref{thm:AttSemiCont}, we're done.
\end{proof}

The following example is elucidating, in that it demonstrates why one needs to consider neighborhoods (distance, in the metric case) in $X\times X$, rather than in $X$.

\begin{ex}\label{Ex:CE1}
Let $X=[0,3].$ Consider $$f=\left([0.8,2+\alpha]\times\{1.5\}\right) \cup \left(\{2+\alpha\}\times[1.5,3] \right).$$ Then $B=[1,2]$ is an attractor block: $$\ol{B}\times\ol{B^c} = [1,2]\times\left([0,1]\cup[2,3]\right),$$ which has no overlap with $f$. It is tempting to consider only the image $f(B)$ and its distance from $B^C$. However, $d\left(\ol{f(B)},\ol{B^c}\right)=0.5=\delta$, while $d\left( f, \ol{B}\times\ol{B^c}\right) = \min\{0.5, \alpha\}$. See Figure \ref{Fig:CE1} for an illustration in which $\alpha = 0.1.$ In order to speak of a neighborhood of $f$, we need to take $\alpha$ into account.

Note also that $\alpha$ can be arbitrarily small without changing the nature of $B$, our attractor block. So, let $\alpha = 0.1<0.5.$ Let us take $\ep$ to be between them: $\alpha<\ep<\delta$. In this case, we can choose $\ep=0.15$. Set
	\[ g = \ol{N_{\ep}}(f)=\ol{N_{0.15}}(f).\]
Then, $g$ is within $0.5$ of $f$, but $g\cap (\ol{B}\times\ol{B^c}) \neq \es$, meaning $B$ fails to be an attractor block for $g$.

\begin{figure}\centering%[h]
\label{Fig:CE1}
  \includegraphics[width=3.3in]{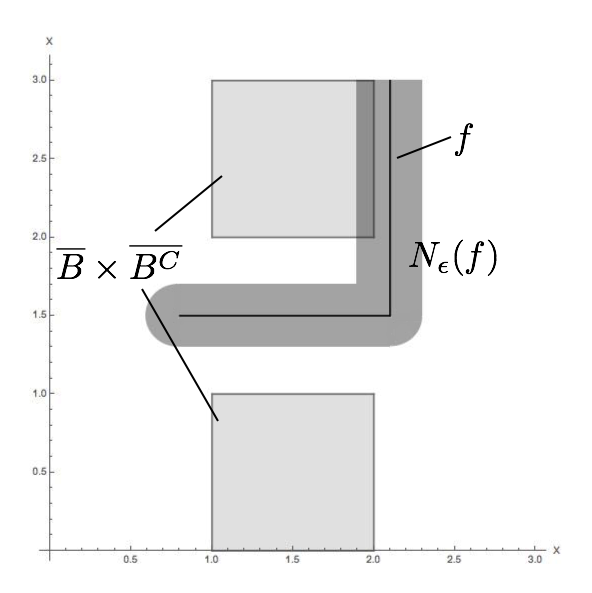}\label{fig:MapEx1f}
  \caption{Illustration of $f$, $N_\ep(f)$, and $\ol{B}\times\ol{B^c}$ from Example \ref{Ex:CE1}.}
\end{figure}

\FloatBarrier

%Scaling $\ep$ based on $d\left( f, \ol{B}\times\ol{B^c}\right)$ won't work, either. The choice of $\alpha$ could have been $10^{-100}$ and we could still choose $\ep$ with $\alpha<\ep<\delta$, without changing the nature of the counterexample.

\end{ex}

\section*{Acknowledgements}
The bulk of this research was completed for my PhD dissertation in the Department of Mathematics at the University of Minnesota. As such, I owe a great debt to my advisor, Dr. Richard McGehee, for his inspiration and guidance, as well as the idea for this problem, which is a continuation of his work in \ref{bk:McGeheeAttractors}. I am grateful for so many wonderful conversations with him. I was able to complete this work, and to generalize the result which occurs in my dissertation, thanks to a postdoctoral fellowship at the Institute for Mathematics and its Applications at the University of Minnesota and funding from Cargill, Inc. %During that time, I was fortunate to benefit even more from discussions with 

\section*{References}
\label{sec:References}
%%%%%%%%%%%%%%%%%%%%%%%%%%%%%%%%%%%%%%%%%%%%%%%%%%%%%%%%%%%%%%%%%%%%%%%%%%%%%%%%
\begin{enumerate}[label={[\arabic*]},itemsep=0mm]%[itemsep=5mm]
%\item P. Collet, J.P. Eckmann, {\it Iterated Maps on the Interval as Dynamical Systems}, Birkh\"{a}user, Boston, (1980).\label{bk:ColletEckmann}
\item C. Conley, {\it Isolated Invariant Sets and the Morse Index}, Reg. Conf. in Math. {\bf 38} CBMS (1978).\label{bk:Conley}
%\item M. Di Bernardo, C.J. Budd, A.R. Champneys \& P. Kowalczyk,
%{\it Piecewise-smooth Dynamical Systems Theory and Applications}, Springer (2008).\label{bk:BBCK}
%\item A. F. Filippov, {\it Differential Equations with Discontinuous Righthand Sides}, Dept. Math., Moscow State University, U.S.S.R., Kluwer Acad. Pub., Boston, MA (1988).\label{bk:Filippov}
%\item J. Guckenheimer, Piecewise-Smooth Dynamical Systems, SIAM Review {\bf 50} (2008), pp. 606-609. \label{bk:Guckenheimer}
%\item J. Leifeld, {\it Smooth and Nonsmooth Bifurcations in Welander's Convection Model}, thesis, University of Minnesota (2016). \label{bk:Leifeld}
\item R. McGehee, {\it Attractors for Closed Relations on Compact Hausdorff Spaces}, IN U. Math. Journal {\bf 41} {\it 4} (1992).\label{bk:McGeheeAttractors}
\item R. McGehee, personal communication (2015-2016).\label{pc:McGehee}
\item R. McGehee, E. Sander, {\it A New Proof of the Stable Manifold Thm}, ZAMP {\bf 74} {\it 4} (1996), pp. 497-513. \label{bk:McGeheeSander}
%\item R. McGehee, C. Thieme, personal communication (2017-2018). \label{pc:Cameron}
\item R. McGehee, T. Wiandt, {\it Conley Decomposition for Closed Relations}, preprint (2005).\label{bk:McGeheeWiandt}
%\item K. J. Meyer, personal communication (2016-2018). \label{pc:KJMeyer}
\item K. Mischaikow, {\it The Conley Theory: A Brief Introduction}, Center for Dyn. Sys. and Nonlinear Studies, Georgia Inst. Tech., Atlanta, GA, Banach Center Publications, Vol **, Inst. of Math., Polish Acad. of Sci. Warszawa {\bf 199*} (1991).\label{bk:Mischaikow}
\item S. Negaard-Paper, {\it Attractors and Attracting Neighborhoods for Multiflows}, PhD Diss. (2019), University of Minnesota, Minneapolis, U.S., \label{bk:Me} {\tt arXiv:1905.06473 [math.FA]}.

\end{enumerate}

\bigskip

This research was supported in part by NSF grants DMS-0940366 and DMS-094036. %I would also like to thank the Institute for Mathematics and its Applications (IMA), where I finished this research.

\end{document}